\documentclass[11pt, pb-diagram]{amsart}\usepackage{pb-diagram} 

\usepackage[utf8]{inputenc} 

\newtheorem{theorem}{Theorem}[section]
\newtheorem{corollary}[theorem]{Corollary}

\newtheorem{lemma}[theorem]{Lemma}
\newtheorem{example}{Example}

\theoremstyle{definition}
\newtheorem{definition}[theorem]{Definition}
\newtheorem{remark}[theorem]{Remark}

\usepackage{graphicx} 

\title{Equivalence of topological dynamics without well-posedness}
\author{Tomoharu Suda}
\address{Faculty of Mathematics, Keio University}
\email{tomoharu.suda@keio.jp}

\begin{document}
\begin{abstract}
The notion of topological equivalence plays an essential role in the study of dynamical systems of flows. However, it is inherently difficult to generalize this concept to systems without well-posedness in the sense of Hadamard. In this study, we formulate a notion of ``topological equivalence" between such systems based on the axiomatic theory of topological dynamics proposed by Yorke, and discuss its relation with the usual definition. During this process, we generalize Yorke's theory to the action of topological groups.
\end{abstract}
\maketitle
\section{Introduction}

Topological classification is a basic problem in the study of dynamical systems. In the case of maps, this task is achieved by considering the topological conjugacy. In the case of flows, it is common to classify systems based on the notion of topological equivalence \cite{robinson1998dynamical, irwin2001smooth, katok1997introduction}. The definition of topological equivalence is usually stated as follows.
\begin{definition}[Topological equivalence]\label{top_equiv}
Let $X$ and $Y$ be topological spaces. Two flows $\Phi : \mathbb{R} \times X \to X$ and $\Psi : \mathbb{R} \times Y \to Y$ are \emph{topologically equivalent} if there exists a homeomorphism $h: X \to Y$ such that each orbit of $\Phi$ is mapped to an orbit of $\Psi$ preserving the orientation of the orbit. 
\end{definition}

 This definition may be rephrased as follows: there exist a homeomorphism $h: X \to X'$ and a map $\tau: \mathbb{R}\times X \to \mathbb{R}$ such that
\begin{equation}\label{equiv_cond}
	\Psi(\tau(t,x), h(x)) = h(\Phi(t,x))
\end{equation}
for all $(t,x) \in \mathbb{R}\times X,$ where $\tau(\cdot, x)$ is monotonically increasing and bijective, and $\tau(0,x) = 0$ for all $x \in X$.

Therefore, the parametrization is ignored with the exception of the orientation. As a result of this property, there is an inherent difficulty in generalizing the notion of topological equivalence to the problem of classification of systems without well-posedness in the sense of Hadamard \cite{tikhonov1977solutions}, as there can be different solutions distinguished only by the parametrization. For example, the following three systems on $\mathbb{R}$ are indistinguishable if we use the same criteria as in Definition \ref{top_equiv}:
\begin{enumerate}
	\item $\dot x \in [1/2,1]$.
	\item $\dot x \in \{1/2,1\}$.
	\item $\dot x = 1.$
\end{enumerate}
Here systems (1) and (2) are differential inclusions. The details on the differential inclusions can be found in \cite{aubin_differential}, for example. 

Thus, in the classification of systems without well-posedness, it is necessary to consider a type of ``topological equivalence," which does not completely ignore the parametrization. 

There are several notions of ``systems" that can be used as the objects of our equivalence. While set-valued flows are often used regarding the consideration of stability, it appears difficult to formulate the condition of orbit preservation in this setting \cite{ball2000continuity, lamb2015topological}. 

A more feasible candidate is the axiomatic theory of the solutions of ordinary differential equations. In particular, Filippov's formalism and Yorke's formalism have potential as the foundation of our methods, although neither considers the morphism between systems \cite{filippov_basic, yorke1969spaces}. In these theories, the solution set of an ordinary differential equation is considered as a set of partial maps on $\mathbb{R}$ satisfying certain axioms. While Filippov's formalism is fairly developed and applicable to a wide variety of problems, it is based on partial maps with closed domains and requires more axioms, which makes the generalization or the consideration of morphisms relatively difficult. Yorke's formalism is based on partial maps with open domains; only a few axioms are required and it has been announced that basic results of the topological dynamics have been generalized. However, the proofs of these results have apparently not been published, as mentioned in \cite{himmelberg1982note}. 

In this article, we consider Yorke's formalism in the setting of the action of topological groups, which can be seen as a generalization of topological dynamics \cite{abs_top_dyn, vries_elements}. In particular, we define the notion of morphism and thereby obtain a notion of ``topological equivalence" applicable to systems without well-posedness. During this process, we state and prove some generalization of the basic results announced in \cite{yorke1969spaces}. 

The remainder of this article is organized as follows. In Section 2, we define the space of partial maps and introduce Yorke's formalism of topological dynamics based on the star-construction. In Section 3,  we consider the shift map of the partial maps and its relationship with the topological transformation groups and dynamical behavior. In Section 4,  we introduce the notion of morphism of the star-construction and equivalences of systems without well-posedness, and discuss their relation with the usual definition of topological equivalence. In Section 5, we consider some examples to illustrate the formalism developed in this article. Finally, in Section 6, we provide concluding remarks.

\section{Space of partial maps and Yorke's formalism}
In this section, we introduce the basic concepts and constructions used throughout the article. In the following, $X$ is a second-countable metric space and $G$ is a locally compact second-countable Hausdorff topological group. Note that $G$ is metrizable by the Birkhoff-Kakutani theorem. 
\begin{definition}[Partial maps]
A continuous map $\phi: D \to X$ a \emph{partial map from $G$ to $X$} if $D \subset G$ is a nonempty open set.   
 
The set of all partial maps is denoted by $C_p(G, X)$. For each $\phi: D \to X,$ we set $\mathrm{dom}\,{\phi} : =D.$ 

A partial map $\phi\in C_p(G,X)$ with a nonempty connected domain is \emph{maximally defined} if, for all $\psi\in C_p(G,X)$ with a nonempty connected domain, the condition $\mathrm{dom}\, \phi \subset \mathrm{dom}\, \psi$ and $\phi=\psi$ on $\mathrm{dom}\, \phi$ implies $\phi=\psi.$
\end{definition}

Our aim is to topologize $C_p(G,X)$ by introducing the compact-open topology. However, as the domains may vary, a slight modification is needed. The following construction is obtained from Abd-Allah and Brown \cite{abd1980compact}.

Let $\omega \notin X$ and $\hat X:= X \cup \{\omega\}.$ The open sets of $\hat X$ are defined by  $\hat X$ and the open sets of $X.$ Now we define a bijection $\mu : C_p(G,X) \to C(G,\hat X)$, where $C(G,\hat X)$ is the set of all continuous maps from $G$ to $\hat X$ topologized by the usual compact-open topology, by setting
\[
	\mu(\phi)(x) := \begin{cases}
					\phi(x) & (x \in \mathrm{dom}\, \phi)\\
					\omega & (\text{otherwise}).
					\end{cases}
\]
The topology on $C_p(G,X)$ is defined so that it makes $\mu$ a homeomorphism. It has a subbasis given by the sets of the form
\[
	W(K, U ) := \{\phi \in C_p(G,X) \mid K \subset \mathrm{dom}\, \phi , \phi(K) \subset U \}.
\] 

The compact-open topology coincides with the topology of compact convergence. 

\begin{lemma}\label{seq_lemma}
The compact-open topology coincides with the topology of compact convergence. Namely, $C_p(G,X)$ is a second countable space and $\phi_n \to \phi$ in $C_p(G,X)$ as $n \to \infty$ if and only if, for all compact subset $K \subset \mathrm{dom}\, \phi,$ $K\subset \mathrm{dom}\, \phi_n $ for sufficiently large $n$ and $\sup_{t \in K} d(\phi_n(t), \phi(t)) \to 0$ as $n \to \infty$.
\end{lemma}
\begin{proof}
The proof of the second countability of $C_p(G,X)$ can be found in \cite{alvarez2016topological} (Proposition 2.1).
%
Though the rest of the proof is similar to that of the classical one, we present it for the sake of completeness.

Let $\phi_n \to \phi$ in $C_p(G,X)$ as $n \to \infty.$ Fix a compact subset $K \subset \mathrm{dom}\, \phi$ and a positive number $\epsilon >0.$ For each $x \in K,$ we may find a relatively compact neighborhood $U_x$ of $x$ such that 
\[
	\begin{aligned}
		& U_x \subset \mathrm{dom}\, \phi,\\
		&\phi(U_x)  \subset B_{\epsilon/3}(\phi(x)),\\
		&\phi(\bar U_x)  \subset B_{\epsilon/2}(\phi(x)).
	\end{aligned}
\]
Because $K$ is compact, there exist $x_1, x_2, \cdots, x_n \in K$ with $ K \subset \bigcup_{i=1}^{n} U_{x_i}.$ Let $K_i := \bar U_{x_i}\cap K$ and $V_i := B_{\epsilon/2}(\phi(x)).$ Then we have $\phi \in \bigcap_{i=1}^{n} W(K_i, V_i).$ Therefore, $\bigcap_{i=1}^{n} W(K_i, V_i)$ is a neighborhood of $\phi$ in $C_p(G,X)$ and there exists $n_0 \in \mathbb{N}$ such that  $\phi_n \in \bigcap_{i=1}^{n} W(K_i, V_i)$ for all $n \geq n_0.$ 

Fix $x\in K.$ Then $x \in K_i$ for some $i$ and we have for all $n \geq n_0,$
\[
	d(\phi(x),\phi_n(x)) \leq d(\phi(x), \phi(x_i))+d(\phi(x_i), \phi_n(x)) < \frac{\epsilon}{2}+\frac{\epsilon}{2} = \epsilon.
\]
Therefore, we conclude that $\sup_{t \in K} d(\phi_n(t), \phi(t)) \leq \epsilon.$

Conversely, let $\phi_n \to \phi$ in the compact convergence and take a neighborhood $\bigcap_{i=1}^{n} W(K_i, V_i)$ of $\phi.$ For each $i$, there exists $r_i > 0$ with 
\[
	B_{r_i}(\phi(K_i)) \subset V_i.	
\]
Let $r := \min(r_i).$ For each $i$, by the definition of compact convergence, there exists $n_i \in \mathbb{N}$ with
\[
	\begin{aligned}
		&K_i \subset \mathrm{dom}\, \phi_n,\\
		&\sup_{t \in K_i} d(\phi_n(t), \phi(t)) < r
	\end{aligned}
\] 
for all $n \geq n_i.$ If we let $n_0 = \max(n_i),$ we have $\sup_{t \in K_i} d(\phi_n(t), \phi(t)) < r$ for all $n \geq n_0$ and $i = 1,2,\cdots, n.$ Therefore, we deduce for each $t \in K_i,$
\[
	\phi_n(t) \in B_r(\phi(t)) \subset V_i.
\]
Thus, we obtain $\phi_n \in \bigcap_{i=1}^{n} W(K_i, V_i).$ 
\end{proof}
 From the next lemma, we may obtain maximally-defined partial maps.

\begin{lemma}
Let $\phi \in C_p(G,X)$ with a nonempty connected domain. Then, there exists a partial map $\bar \phi \in C_p(G,X)$ which is maximally defined and satisfies $\phi = \bar \phi$ on $\mathrm{dom}\, \phi$.
\end{lemma}
\begin{proof}
We define an order on the set $C_c:=\{ \phi \in C_p(G,X)\mid \mathrm{dom}\, \phi \text{ is connected}\}$ by setting $\phi \leq \psi$ if $\mathrm{dom}\, \phi \subset \mathrm{dom}\, \psi$ and $\phi(x) = \psi(x)$ for all $x \in \mathrm{dom}\, \phi.$ This is clearly a partial order on $C_c.$

Fix $\phi \in C_c$ and consider $\mathcal{C}_\phi := \{\psi \in C_c\mid \phi \leq \psi\}.$ As $\phi \in \mathcal{C}_\phi,$ $\mathcal{C}_\phi$ is nonempty. Let $\mathcal{S}$ be a chain in $\mathcal{C}_\phi$. We define a map $\bigvee \mathcal{S}$ by setting
\[
\begin{aligned}
	\mathrm{dom}\, \bigvee \mathcal{S} &:= \bigcup_{\psi \in \mathcal{S}} \mathrm{dom}\, \psi\\
	\bigvee \mathcal{S}(x) &:= \psi(x) \text{ if } x \in \mathrm{dom}\, \psi.
\end{aligned}
\]
Because $\mathcal{S}$ is totally ordered, $\bigvee \mathcal{S}$ is well-defined, $\bigvee \mathcal{S} \in \mathcal{S},$ and $\psi \leq \bigvee \mathcal{S} $ for all $\psi \in \mathcal{S}$. Therefore, from Zorn's lemma, we see that $\mathcal{C}_\phi$ has a maximal element, which is the desired partial map.
\end{proof}
 
Maximally defined partial maps have properties similar to those of the solution curves of an ordinary differential equation. The set of all maximally defined partial maps is denoted by $C_s(G, X)$. Clearly, $C(G,X)$ with compact-open topology is a subspace of $C_s(G, X)$.

We define the \emph{orbit} $\mathcal{O}(\phi)$ of $\phi \in C_s(G, X)$ by
\[
	\mathcal{O}(\phi) := \{\phi(g) \mid g \in \mathrm{dom}(\phi)\}.
\]

In general, the space $C_s(G, X)$ is not Hausdorff and the limit of a sequence is not unique. However, we have the following property.
\begin{lemma}\label{lem_conv_unique}
Let $\{\phi_n \}_n$ be a convergent sequence  in  $C(G, X)$ with $\phi_n \to \phi$ as $n \to \infty.$
If $\phi_n \to \psi$ as $n \to \infty $ in  $C_s(G, X),$ then $\phi = \psi.$

Consequently, if $\{\phi_{n_i}\}$ is a subsequence of $\{\phi_n \}_n$ convergent in $C_s(G, X),$ $\phi_{n_i} \to \phi$ as $i \to \infty$ in $C_s(G, X).$
\end{lemma} 
\begin{proof}
By Lemma \ref{seq_lemma}, we have $\phi = \psi$ on $\mathrm{dom}(\psi) \subset G = \mathrm{dom}(\phi).$ Because $\psi \in C_s(G,X),$ we have $\phi = \psi.$
\end{proof}
Further, the next construction by Yorke enables us to avoid the difficulty accompanied with $C_s(G,X)$ \cite{yorke1969spaces}.

\begin{definition}[Star-construction]
For a subset $S \subset C_s(G,X)$ and $W \subset G \times X,$ we define
\[
	S^*W := \{(g,\phi) \mid \phi \in S, g \in \mathrm{dom}\,\phi, (g, \phi(g) )\in W\}.
\]
Because $S^*W \subset G \times C_s(G,X)$, the topology of $S^*W$ is naturally defined by the subspace topology. We call $S^*W$ the \emph{star-construction} defined by $S$ and $W.$
\end{definition}
Basic properties of the star-construction may be summarized as follows:

\begin{lemma}\label{basic_rels}
The following properties hold:
\begin{enumerate}
	\item For $S, S'\subset C_s(G,X)$ and $W, W' \subset G \times X$,$S \subset S'$ and $W\subset W'$ implies $S^*W \subset S'^* W'$ and $S^*(W'\backslash W) = S^*W' \backslash S^*W.$
	\item For $S\subset C_s(G,X)$ and $ W_\mu \subset G \times X$ ($\mu \in M$), we have
\[
	S^*\bigcup_{\mu \in M} W_\mu = \bigcup_{\mu \in M} S^*W_\mu, 
\]
and
\[
	S^*\bigcap_{\mu \in M} W_\mu = \bigcap_{\mu \in M} S^*W_\mu.
\]
	\item For subsets $S_\lambda\subset C_s(G,X)$ ($\lambda \in \Lambda$) and $W \subset G \times X$, we have

\[
	(\bigcup_{\lambda \in \Lambda} S_\lambda)^* W = \bigcup_{\lambda \in \Lambda} S_\lambda^* W, 
\]
and
\[
	(\bigcap_{\lambda \in \Lambda} S_\lambda)^* W = \bigcap_{\lambda \in \Lambda} S_\lambda^* W.
\]
\end{enumerate}
\end{lemma}
\begin{proof}
The proofs are straightforward.
\end{proof}

The next theorem enables us to analyze the ``solution space" $S^*W$ in a way similar to that of the usual ordinary differential equations.
\begin{theorem}\label{top_thm}
For all $S \subset C_s(G,X)$ and $W \subset G \times X$, $S^*W$ is second countable and Hausdorff.
\end{theorem}
\begin{proof}
$G \times C_s(G,X)$ is a second countable space by  Lemma \ref{seq_lemma}. Consequently, $S^*W$ is also second countable. 

For the Hausdorff property, let $(g, \phi), (g', \psi) \in S^*W$ with  $(g, \phi) \neq (g', \psi).$  

If $g \neq g',$ we may find open neighborhoods $U_g, U_{g'} \subset G$ of $g $ and $ g'$ with $U_g \cap U_{g'} = \emptyset.$ Then we have $U_g \times V \cap U_{g'} \times V'= \emptyset$ for all open neighborhoods $V,V'$ of $\phi$ and $\psi.$

If $\phi \neq \psi,$ the maximality of domains implies that $\mathrm{dom}\,\phi \cap \mathrm{dom}\,\psi = \emptyset$ or $\phi(g_0) \neq \psi(g_0)$ for some $g_0 \in \mathrm{dom}\,\phi \cap \mathrm{dom}\,\psi.$ In the former case we have $g \neq g',$ which we have already considered.  In the latter case, there exists $r > 0$ with $B_r(\phi(g_0) )\cap B_r(\psi(g_0)) = \emptyset.$ Because $W \left(\{g_0\}, B_r(\phi(g_0) )\right)$ and  $W(\{g_0\}, B_r\left(\psi(g_0) )\right)$ are disjoint, we have $U \times W \left(\{g_0\}, B_r(\phi(g_0) )\right) \cap U' \times W \left(\{g_0\}, B_r(\psi(g_0) )\right) =\emptyset $ for all neighborhoods $U$ and $U'$ of $g$ and $g'.$
\end{proof}
\begin{remark}\label{ch_star}
The star construction may be characterized as follows. We establish the evaluation map
$\mathrm{ev} :  G\times C_s(G,X)\to G \times \hat X$
by $\mathrm{ev}(g, \phi):= (g, \mu(\phi)(g)),$ and the projection
$\mathrm{p} :  G\times C_s(G,X)\to C_s(G,X)$
by $\mathrm{p}(g, \phi):= \phi.$ Note that these maps are continuous. Then the set $S^* W$ is the largest subset $E$ of $G\times C_s(G,X)$ such that $\mathrm{ev}(E) \subset W $ and $\mathrm{p}(E) \subset S.$  In this sense, the set $S^* W$ characterizes the ``Cauchy problem" on $W$ with the solution set $S$, which amounts to finding a map $\phi \in S$ satisfying $\phi(g) = x$ for each $(g,x) \in W$.
\end{remark}
In the next definition, we list the main additional axioms proposed by Yorke, which are abstractions of the conditions for well-posedness.
\begin{definition}\label{def_props}
Let $S \subset C_s(G,X).$
\begin{enumerate}
	\item The subset $S$ satisfies the \emph{compactness axiom} if $S^* W$ is compact for all compact $W \subset G \times X.$
	\item The subset $S$ satisfies the \emph{existence axiom} on $W$ if $S^*\{(t,x)\}$ is nonempty for all $(t, x) \in W$.
	\item The subset $S$ satisfies the \emph{uniqueness axiom} on $W$ if $S^*\{(t,x)\}$ is empty or a singleton for all $(t, x) \in W$.
	\item The subset $S$ has \emph{domain} $D$ if $D \subset \mathrm{dom}\,\phi$ for all $\phi \in S$.
\end{enumerate}
\end{definition}
\begin{remark}
Because $S^* W$ is second countable, $S$ satisfies the compactness axiom if and only if $S^* W$  is sequentially compact for all compact $W \subset G \times X.$
\end{remark}

The next result, which is an analogue of the classical convergence theorem (Chapter 2, Theorem 3.2 in \cite{hartman2002ordinary}), clarifies that the compactness axiom is an abstraction of the continuous dependence on the initial conditions.
\begin{theorem}\label{kamke}
A subset $S \subset C_s(G,X)$ satisfies the compactness axiom if and only if the following property holds: if $(g_n, x_n) \to (g, x)$ in $G \times X$ and there exists a sequence of maps $\phi_n \in S$ with $\phi_n(g_n) = x_n,$ there is a map $\psi \in S$ with $\psi(g) = x$ and a subsequence $\{(g_{n_i} \phi_{n_i})\}$ with $(g_{n_i}, \phi_{n_i}) \to (g, \psi).$
\end{theorem}
\begin{proof}
Let $S \subset C_p(G,X)$ satisfy the compactness axiom and $(g_n, x_n) \to (g, x)$ be a sequence in $G \times X$ with a corresponding sequence of maps $\phi_n \in S$ satisfying $\phi_n(g_n) = x_n.$ Then $W := \{(g,x) \}\cup\{(g_n, x_n)\mid n \in \mathbb{N} \}$ is compact, and therefore, $S^*W$ is sequentially compact. Because $(g_n, \phi_n) \in S^*W$ for each $n,$ we may take a convergent subsequence $(g_{n_i},\phi_{n_i}) \to (g', \psi) \in S^*W.$ It is clear that $g = g'.$ 

Conversely, let $W \subset G \times X$ compact and $(g_n,\phi_n)$ ($n=1,2,\cdots $) be a sequence in $S^*W.$ Then, $(g_n, \phi_n(g_n)) \in W$ for $n=1,2,\cdots $, and therefore, we can find a convergent subsequence $(g_{n_i}, \phi_{n_i}(g_{n_i})) \to (g, y) \in W$ ($i\to \infty$). From the hypothesis, there exist a subsequence $(g_{n_{i_j}}, \phi_{n_{i_j}})$ and a pair $(g,\psi) \in S^* W $ satisfying $(g_{n_{i_j}}, \phi_{n_{i_j}}) \to (g, \psi)$ as $j \to \infty$ and $\psi(g) = y.$ Therefore, $S^* W$ is sequentially compact. 
\end{proof}
From the classical convergence theorem, we immediately observe that the solution space of an ordinary differential equation satisfies the compactness axiom:
\begin{corollary}
Consider an ordinary differential equation
	\[x' = f(t,x),\]
where $f:\mathbb{R} \times W \to W$ is continuous and $W \subset \mathbb{R}^n$ is a nonempty open set.
	Then the solution space $S$ satisfies the compactness axiom.
\end{corollary}
Further, a topological transformation group can be identified with a function space satisfying the compactness axiom. Details on the topological transformation groups may be found in \cite{bredon1972introduction}.
\begin{corollary}\label{action_set}
	Let $\pi : G \times X \to X$ be a continuous left $G$-action on $X.$ Then the set of maps \[S := \{\pi(\cdot,x) \mid x \in X\}\]
	satisfies the compactness, existence, and uniqueness axioms if $X$ is locally compact.
\end{corollary}
\begin{proof}	
The existence and uniqueness axioms are satisfied because we have
\[
	S^*\{(t,x)\} = \left\{\left(t, \pi\left( \cdot, \pi(t^{-1}, x)\right)\right)\right\}
\]
for each $(t,x) \in G\times X$.

For the compactness axiom, we show that $\pi\left( \cdot, \pi(t_n^{-1}, x_n)\right) \to  \pi\left( \cdot, \pi(t^{-1}, x)\right)$ as $n \to \infty$ whenever $(t_n, x_n) \to (t, x)$ as $n \to \infty$. Let $y_n := \pi(t_n^{-1}, x_n)$ and $y := \pi(t^{-1}, x),$ and fix a compact subset $K \subset G.$ We may find a neighborhood $U$ of $y$ with $\bar U$ being compact and $y_n \in U$ for sufficiently large $n$. Because $\pi$ is continuous, $\pi$ is uniformly continuous on $K \times \bar U.$ Therefore, we have $\sup_{s\in K} d(\pi(s, y_n), \pi(s, y)) \to 0$ as $n \to \infty.$
\end{proof}
In some cases, the compactness axiom is not independent with other axioms and certain restrictions are present.
\begin{corollary}
Let subset $S \subset C_s(G,X)$ satisfy the compactness axiom and there exists a subset $W \subset G \times X$ such that $S$ satisfies the existence axiom on $W$. If $X$ is locally compact, then $S$ satisfies the existence axiom on the closure of $W$.
\end{corollary}
The following theorem is a generalization of a well-known result on the global solution.
\begin{theorem}
Let $S \subset C_s(G,X)$ satisfy the compactness axiom and $\phi \in S.$ If there is a compact set $K$ with $\mathcal{O}(\phi) \subset K,$ $\mathrm{dom}\, \phi$ coincides with a connected component of $G.$ 
\end{theorem}
\begin{proof}
There is a connected component $G_c$ of $G$ satisfying $\mathrm{dom}\,\phi\subset G_c.$ If $\mathrm{dom}\, \phi \neq G_c,$ $\partial \left( \mathrm{dom}\, \phi \right) \neq \emptyset.$ Then we may take a sequence $g_n \to g \in \partial \left( \mathrm{dom}\, \phi \right)$ with $g_n \in \mathrm{dom}\, \phi.$ Because the image of $\phi$ is contained in $K,$ we may assume $\phi(g_n) \to y \in K.$ By Theorem \ref{kamke}, there exists $(g, \psi)$ with $(g_{n_i}, \phi) \to (g, \psi).$ Therefore, $g \in \mathrm{dom}\, \phi,$ which is a contradiction.
\end{proof}
The next theorem is a generalization of a result announced by Yorke.
\begin{theorem}
If $S \subset C_s(G,X)$ satisfies the compactness axiom and $X$ is locally compact, $S^* W$ is metrizable for each subset $W \subset G \times X$.
\end{theorem}
\begin{proof}
First we show that  $S^* W$ is locally compact. Fix $(g,\phi) \in S^*W.$ Because $G$ and $X$ are locally compact, we may take relatively compact neighborhoods $U$ and $V$ of $g $ and $\phi(g)$ such that $\phi(\bar U) \subset V$. Then $S^*(\bar U \times \bar V) \cap S^*W$ is a compact neighborhood of $(g, \phi)$ in $S^*W.$ This is verified by observing
\[
 (g, \phi) \in U \times W(\bar U , V) \cap S^*W  \subset S^*(\bar U \times \bar V) \cap S^*W.
\] 

Because $S^* W$ is locally compact and Hausdorff, it is regular.  Using the second countability of $S^* W$ (Theorem \ref{top_thm}), we may apply the Urysohn metrization theorem.
\end{proof}
\section{Shift of partial maps and dynamical behavior}
Based on the concepts introduced in the previous section, we consider generalizations of the usual notions in dynamical systems theory. 

The concept of shift invariance plays a central role in the discussion in this section. First we establish that the shift map is a $G$-action.
\begin{theorem}[Shift map]\label{shift_thm}
The shift map $\sigma : G \times C_p(G,X) \to C_p(G,X) $, which is defined by
\[
	\sigma(g, \phi)(x) := \phi(x g)
\]
for $x \in \mathrm{dom}(\phi)g^{-1}$, is continuous and satisfies the following conditions:
\begin{enumerate}
	\item For each $\phi\in C_p(G,X)$ we have $\sigma(e, \phi) = \phi.$
	\item For all $g,h \in G$ and $\phi\in C_p(G,X)$, we have $\sigma(g,\sigma(h,\phi)) = \sigma(gh, \phi).$
\end{enumerate}
That is, $\sigma$ is a left $G$-action.
\end{theorem}
\begin{proof}
To show that $\sigma$ is continuous, it suffices to confirm that $\sigma^{-1}\left( W(K,V)\right)$ is open for each compact $K \subset G$ and open $V \subset X.$ Let $(g, \phi) \in \sigma^{-1}\left( W(K,V)\right).$ This is equivalent to
\[	
	\sigma(g,\phi)(K) = \phi\left(K g\right) \subset V.
\]
Therefore, we have $\phi \in W\left(K g, V\right).$ We may find an open neighborhood $U$ of $g$  
such that $\bar U$ is compact and
\[
	K \bar U \subset \mathrm{dom}(\phi) \cap \phi^{-1}(V),
\]
using the regularity of $G$. Then we have
\[
	(g, \phi) \in U \times W\left(K \bar U, V \right) \subset \sigma^{-1}\left( W(K,V)\right).
\]
Thus, $\sigma^{-1}\left( W(K,V)\right)$ is open.

The other assertions are proven via direct calculations.
\end{proof}
\begin{remark}
By the definition of the shift map, we have
\[
	\mathrm{dom}\left( \sigma(g, \phi)\right) = \mathrm{dom}\left(\phi\right)g^{-1}
\]
for each $(g,\phi) \in G \times C_p(G,X).$ Further, the following identity holds for all left $G$-action $\pi : G\times X \to X:$
\begin{equation}\label{id_action}
	\sigma\left(g, \pi(\cdot, x) \right) = \pi\left(\cdot,  \pi(g, x)\right).
\end{equation}
\end{remark}

Thus, the triplet $(G, \sigma, C_p(G,X))$ is a topological transformation group. 

As the name ``shift system" usually refers to discrete systems, here we call it the \emph{Bebutov system} on $C_p(G,X)$.

Now we clarify the relationship between the Bebutov system and the usual topological transformation groups.
This is given by the next theorem, which is a generalization of a result by Yorke (Theorem 2.3 in \cite{yorke1969spaces}). 
\begin{theorem}\label{rep_flow}
Let $X$ be locally compact. Then a $\sigma$-invariant subset $S \subset C_s(G,X)$ satisfies the compactness, existence, and uniqueness axioms and has domain $G$ if and only if it is given by a left $G$-action $\pi_S: G\times X \to X$ on $X$ via 
\begin{equation}
S := \{\pi_S(\cdot,x) \mid x \in X\}.
\end{equation}
\end{theorem}
\begin{proof}
The sufficiency is a consequence of Corollary \ref{action_set} and the identity (\ref{id_action}). 

Let a $\sigma$-invariant subset $S \subset C_s(G,X)$ satisfy the compactness, existence, and uniqueness axioms and have domain $G$. We establish a map $\pi_S: G \times X \to X$ by 
\[
	\pi_S(g,x) : = \phi_{(e,x)} (g),
\]
where $\phi_{(e,x)} $ is the unique element in $S$ with $\phi_{(e,x)}(e) = x.$ We note that
$\sigma \left(g, \phi_{(e,x)} \right) \in S$ and 
\[\sigma \left(g, \phi_{(e,x)} \right)(e) = \phi_{(e,x)} (g) = \pi_S(g,x).\]
Therefore, we have
\begin{equation}\label{conj_eq}
	\phi_{\left(e, \pi_S(g,x)\right)}  = \sigma \left(g, \phi_{(e,x)} \right).
\end{equation}
From this identity, we obtain
\[
\begin{aligned}
	\pi_S\left(g, \pi_S(h,x) \right) &= \phi_{\left(e, \pi_S(h,x)\right)}(g)\\
							&= \sigma \left(h, \phi_{(e,x)} \right)(g)\\
							&=  \sigma \left( g, \sigma \left(h, \phi_{(e,x)} \right) \right)(e)\\
							&= \sigma \left( gh, \phi_{(e,x)} \right)(e)\\
							&= \phi_{(e,x)}(gh) \\
							&= \pi_S(gh,x).
	\end{aligned}
\]
Finally, we show that $\pi_S$ is continuous. Let $(g_n, x_n) \to (g,x)$ as $n \to \infty$ in $G \times X.$ From Theorem \ref{kamke} and the uniqueness, we have $\phi_{(e,x_n)} \to \phi_{(e,x)}$ as $n \to \infty$ in $C_s(G,X).$ By Theorem \ref{shift_thm}, we obtain
\[
	\sigma\left(g_n, \phi_{(e,x_n)}\right) \to \sigma\left(g, \phi_{(e,x)}\right)
\]
in $C_s(G,X).$ Evaluation at $e$ gives us $\pi_S(g_n,x_n) \to \pi_S(g,x) $ in $X$.
\end{proof}
\begin{corollary}
Let $X$ be locally compact. If a $\sigma$-invariant subset $S \subset C_s(G,X)$ satisfies the compactness, existence, and uniqueness axioms and has domain $G$, $S$ is homeomorphic to $X$. Furthermore, $(S, \sigma)$ and $(X, \pi_S)$ are isomorphic as topological transformation groups.
\end{corollary}
\begin{proof}
Let us consider the map $p:S \to X$ defined by
\[
	p(\phi) = \phi(e)
\]
and $s: X \to S$ defined by
\[
	s(x) = \phi_{(e,x)}.
\]
The continuity of $p$ follows from Lemma \ref{seq_lemma} and that of $s$ from the identity $s(x) = \pi_S(\cdot, x).$ Clearly, $p\circ s = \mathrm{id}_X$ and $s \circ p = \mathrm{id}_S.$

The identity (\ref{conj_eq}) can be rewritten as
\[
	s(\pi_S(g,x)) =  \sigma \left(g, s(x) \right).
\]
Further, the following identity holds:
\[
	 p\left(\sigma \left(g, \phi \right)\right) =  \pi_S(g, p(\phi)).
\]
Hence, $(S, \sigma)$ and $(X, \pi_S)$ are isomorphic as topological transformation groups.
\end{proof}

Thus, each topological transformation groups on $X$ can be interpreted as a $\sigma$-invariant subset of $C_s(G,X),$ and its asymptotic behavior can be stated in terms of the functional space $C_s(G,X).$

Notions regarding the dynamical behavior can be generalized to a $\sigma$-invariant subset of $C_s(G,X).$ First we note the following property of the ``initial value problem," which simplifies the definitions. 
\begin{lemma}\label{lem_iv_ch}
Let $S$ be a $\sigma$-invariant subset of $C_s(G,X)$ and $x \in X.$ If $x = \phi(g)$ for some $\phi \in S$ and $g \in \mathrm{dom}(\phi),$ there exists $\hat \phi \in S$ with $\hat\phi(e) = x.$
\end{lemma}
\begin{proof}
Set $\hat \phi  = \sigma(g, \phi) \in S.$
\end{proof}
\begin{definition}[Weak invariance, equilibrium point]
For a $\sigma$-invariant subset $S$ of $C_s(G,X),$ we define the following:
\begin{enumerate}
	\item A subset $A \subset X$ is \emph{weakly invariant} with respect to $S$ if, for each $x \in A,$ there exists $(e, \phi) \in S^*\{(e,x)\}$ with $\mathcal{O}(\phi) \subset A.$
	\item A point $x \in X$ is an \emph{equilibrium} of $S$ if there exists $(e, \phi) \in S^*\{(e,x)\}$ with $\mathcal{O}(\phi) = \{x\}.$
%
	\end{enumerate}

\end{definition}
The following are some generalizations of well-known properties of these notions.
\begin{lemma}\label{orb_inv_lem}
Let $S$  be a  $\sigma$-invariant subset of $C_s(G,X)$. For each $\phi \in S,$ $\mathcal{O}(\phi)$ is weakly invariant.
\end{lemma}
\begin{proof}
Let $x \in \mathcal{O}(\phi).$ By definition, there exists $g \in \mathrm{dom}(\phi)$ with $\phi(g) =x.$ Then we have $(e, \sigma(g, \phi)) \in S^*\{(e,x)\}.$ Further, $\mathcal{O}(\sigma(g, \phi)) = \mathcal{O}(\phi).$ Hence, $\mathcal{O}(\phi)$ is weakly invariant.
\end{proof}
\begin{lemma}\label{cl_inv_lem}
Let $S$  be a  $\sigma$-invariant subset of $C_s(G,X)$ satisfying the compactness axiom and $X$ be locally compact. If $A \subset X$ is compact and weakly invariant with respect to $S$, $\bar A$ is also weakly invariant.
\end{lemma}
\begin{proof}
Let $x \in \bar A.$ Then we may take a sequence $\{x_n \} \subset A$ with $x_n \to x$ as $n \to \infty.$ Because $A$ is weakly invariant, we may take, for each $n$, $\phi_n \in S$ with $\phi_n(e) = x_n$ and $\mathcal{O}(\phi_n) \subset A.$ Let $U$ be a relatively compact open neighborhood of $x$. For sufficiently large $n$, we have
\[
	(e, \phi_n) \in S^*(\{e\}\times \bar U).
\]
By the compactness of $S$, we may assume that there exists $(e, \phi) \in S^*(\{e\}\times \bar U)$ with $(e, \phi_n) \to (e, \phi)$ as $n \to \infty.$ Therefore, $\phi(e) = \lim_{n \to \infty} \phi_n(e) = x.$ Further, for each $g \in \mathrm{dom}(\phi),$ we have 
\[
	\phi(g) = \lim_{n\to \infty} \phi_n(g). 
\]
Because $\phi_n(g) \in A,$ $\phi(g) \in \bar A.$
\end{proof}

Note that an intersection of weakly invariant sets need not be weakly invariant due to the lack of uniqueness.

%
%

\section{Morphisms of the star-construction and topological equivalence}
As mentioned in Remark \ref{ch_star}, a star-construction can be characterized in terms of two maps
$\mathrm{ev}:G \times C_s(G,X) \to G \times X$ and $\mathrm{p}: G \times C_s(G,X) \to C_s(G,X) $. Morphisms of the star-construction can be defined so that they respect this structure.
\begin{definition}[Morphisms of the star-construction]
Let $S \subset C_s(G,X)$,$S' \subset C_s(G',X'),$ $W \subset G \times X$ and $W' \subset G' \times X'$.  A \emph{morphism} between the star-constructions $S^*W$ and $(S')^* W'$ is a triplet of continuous maps
\[
	\begin{aligned}
		&H : S^*W \to (S')^* W'\\
		&k: W \to W'\\
		&\eta: S \to S' 
	\end{aligned}
\]
such that 
\[
	\begin{aligned}
		\mathrm{ev} \circ H &=  k \circ \mathrm{ev},\\
		\mathrm{p} \circ H &=  \eta \circ \mathrm{p}.
	\end{aligned}
\]
We denote a morphism by $<H, k, \eta>: S^*W \to (S')^* W'.$

If there exists a morphism $<H, k, \eta>: S^*W \to (S')^* W'$ such that $H,k$ and $\eta$ are homeomorphisms, then $<H, k, \eta>$ is an \emph{isomorphism} and $S^*W$ and $(S')^* W'$ are \emph{isomorphic}.
\end{definition}

The above definition of isomorphism is justified by the following lemma.
\begin{lemma}
Let $S \subset C_s(G,X)$, $S' \subset C_s(G',X'),$ $S'' \subset C_s(G'',X''),$ $W \subset G \times X, $ $W' \subset G' \times X'$ and $W'' \subset G'' \times X''$. If $<H, k, \eta>: S^*W \to (S')^* W'$ and $<H', k', \eta'>: (S')^*W' \to (S'')^* W''$ are morphisms, then $<H' \circ H, k' \circ k, \eta' \circ \eta>: S^*W \to (S'')^* W''$ is also a morphism.

If $H, k, \eta$ are homeomorphisms, then $<H^{-1}, k^{-1}, \eta^{-1}>: (S')^* W' \to S^*W $ is also a morphism.
\end{lemma}
\begin{lemma}\label{inc_star_lem}
Let$ <H, k, \eta>: S^*W \to (S')^* W'$ be a morphism. Then we have
\[ H\left(A^* B\right) \subset \eta(A)^*k(B) \subset (S')^* W'\]
for each $A \subset S$ and $B \subset W.$ In particular, if $ <H, k, \eta>: S^*W \to (S')^* W'$ is an isomorphism, $A^*B$ and $\eta(A)^*k(B)$ are isomorphic.
\end{lemma}
\begin{proof}
By Lemma \ref{basic_rels}, $A^* B \subset S^*W.$ $H\left(A^* B\right) \subset \eta(A)^*k(B) $ is an immediate consequence of the definition of morphisms and Remark \ref{ch_star}.
\end{proof}

The axioms listed in Definition \ref{def_props} are preserved by isomorphisms.
\begin{theorem}\label{iso_inv}
Let $S \subset C_s(G,X)$ and $S' \subset C_s(G',X').$ If $S^*(G\times X)$ and $(S')^*(G'\times X')$ are isomorphic by $<H, k, \eta>$, the following assertions are true.

\begin{enumerate}
	\item The subset $S$ satisfies the compactness axiom if and only if $S'$ does so.
	\item The subset $S$ satisfies the existence axiom on $W$ if and only if $S'$ does so on $k(W).$
	\item The subset $S$ satisfies the uniqueness axiom on $W$ if and only if $S'$ does so on $k(W).$
\end{enumerate}
\end{theorem}
\begin{proof}
From Lemma \ref{inc_star_lem}, we have
\[
	\begin{aligned}
		H\left( S^* W\right) &= \eta(S)^*k(W)\\
						&= (S')^* k(W).
	\end{aligned}
\]
Because $H$ is a homeomorphism, we obtain the results.
\end{proof}

The equivalence class of subsets of $C_s(G,X)$ under the isomorphism relation is rather large. This classification can be regarded as that of the types of problems, as observed by the following result.
\begin{theorem}
Let $X$ be locally compact. If a $\sigma$-invariant subset $S \subset C_s(G,X)$ satisfies the compactness, existence, and uniqueness axioms and has domain $G$, $S^*(G \times X)$ is isomorphic to $S_0^*(G \times X),$ where
\[
	S_0 := \{\psi_x \in C_s(G,X) \mid \psi_x(g) = x \text{ for all } g \in G\}.
\]
\end{theorem}
\begin{proof}
By Theorem \ref{rep_flow}, there exists a continuous map $\pi_S : G \times X \to X$ such that
\[
S = \{\pi_S(\cdot,x) \mid x \in X\}.
\]
We establish a morphism $<H, k, \eta>: S_0^*(G \times X) \to S^*(G \times X)$ by 
\[
	\begin{aligned}
		k(g, x) &:= (g, \pi_S(g,x))\\ 
		\eta(\psi) &:= \pi_S(\cdot, \psi(e))\\
		H(g, \psi) &:=  (g, \eta(\psi)).
	\end{aligned}
\]
Because $k^{-1}(g,x) = (g, \pi_S(g^{-1},x))$ and $\eta^{-1}(\pi_S(\cdot,x)) =  \psi_x,$ $H,k$ and $\eta$ are homeomorphisms. Therefore, $S^*(G \times X)$ is isomorphic to $S_0^*(G \times X).$
\end{proof}
It is apparent from the preceding theorem that an isomorphism may mix spatio-temporal structure. In particular, weak invariance is not respected for $\sigma$-invariant sets.
Considering this point,  a more useful notion is defined as follows.
\begin{definition}[Phase space-preserving morphism]
Let $S \subset C_s(G,X)$,$S' \subset C_s(G',X'),$ $W \subset G \times X$ and $W' \subset G' \times X'$.  A morphism $<H,k, \eta>$ between the star-constructions $S^*W$ and $(S')^* W'$ \emph{preserves phase space} if $k$ has the form $k=(\tau, h),$ where $\tau: W \to G'$ and $h: W \to X'$ are continuous and $h(g,x) = h(g' ,x)$ for all $(g, x), (g',x) \in W.$ 
 
 $S^*W$ and $(S')^* W'$ are \emph{isomorphic via phase space-preserving isomorphisms} if there exists a phase space-preserving isomorphism $<H,k, \eta>$ between $S^*W$ and $(S')^* W'$ such that    $<H^{-1},k^{-1}, \eta^{-1}>$ also preserves phase space.
\end{definition}
\begin{remark}\label{rem_equiv}
The identity morphism is a phase space-preserving isomorphism because we have $k =\mathrm{id}_{W} = (p_G, p_X)$ where $p_G: W \to G$ and $p_X: W \to X$ are canonical projections. Also, the composition of two phase space-preserving isomorphisms preserves phase space. Therefore, the relation of isomorphism via phase space-preserving isomorphisms is an equivalence relation.
\end{remark}

If a morphism preserves phase space and $W$ is open, there is a map induced on the ``phase space."
\begin{lemma}
Let $S \subset C_s(G,X)$, $S' \subset C_s(G',X'),$ $W \subset G \times X$ and $W' \subset G' \times X'$. Let $\hat W := \{ x \in X \mid (g,x) \in W \text{ for some } g \in G\}.$ If a morphism $<H,k, \eta>$ preserves phase space and $W$ is open in $G \times X$, there exists a continuous map $\hat h: \hat W \to X'$ defined by setting $\hat h(x) := h(g,x)$ for some $g \in W_x := \{g \in G \mid (g, x) \in W\}.$ 
\end{lemma}
\begin{proof}
Let the canonical projection $p_X: W \to \hat W$ be defined by $p_X(g,x) : = x.$ It suffices to show that $p_X$ is an identification map because we have the following commutative diagram (for details, see Ch 4, Theorem 3.1 in \cite{james1966topology}):
\[
\begin{diagram}
\node{W} 
\arrow{s,r}{p_X}
\arrow{se,t}{h} 
\\
\node {\hat W} 
\arrow{e,b}{\hat{h}} 
\node{X'} 
\end{diagram}
\]
Let $\hat U \subset \hat W$ be open in $\hat W$. Then there exists an open subset $U \subset X$ such that $\hat U = U \cap \hat W.$ Because $p_X^{-1}(\hat U) = (G \times U) \cap W,$ $p_X^{-1}(\hat U)$ is open in $W.$

Conversely, let $\hat U \subset \hat W$  and $p_X^{-1}(\hat U)$ be open in $W.$ Then there exists an open subset $V \subset G \times X$ such that $p_X^{-1}(\hat U) = V \cap W.$ For each $x \in \hat U,$ we may find $g \in G$ with $(g,x) \in V \cap W.$ Because $V \cap W$ is open, there exist open subsets $U_1 \subset G$ and $U_2 \subset X$ with $(g, x) \in U_1 \times U_2 \subset V \cap W.$ 

Now we show that $U_2 \cap \hat W \subset \hat U.$ For each $x' \in U_2 \cap \hat W,$  we have $(g, x') \in U_1 \times U_2$ for all $g \in U_1.$ Therefore, $(g, x') \in V \cap W = p_X^{-1}(\hat U).$ By definition, $x' = p_X(g, x') \in \hat U.$
\end{proof}
\begin{remark}
Because $k(g,x) = (\tau(g, x), \hat h(x)) \in W'$ holds for some $g \in G$ for each $x \in \hat W$, we have $\hat h: \hat W \to \hat W' := \{ x' \in X' \mid (g',x') \in W' \text{ for some } g' \in G'\}.$ 
\end{remark}
\begin{remark}
For a phase space-preserving isomorphism $<H,k, \eta>$ between $S^*W$ and $(S')^* W'$, $<H^{-1},k^{-1}, \eta^{-1}>$ also preserves phase space if and only if $\hat h$ is injective.
\end{remark}

The property in the following lemma is crucial in the discussion below.
\begin{lemma}\label{iv_lem}
Let $S \subset C_s(G,X)$, $S' \subset C_s(G',X'),$ $W \subset G \times X$ and $W' \subset G' \times X'$. If a morphism $<H,k, \eta>$ preserves phase space, then we have
\[
	\eta(\phi)(\tau(g, x)) = \hat h(x),
\]
for each $(g, \phi) \in S^*\{(g,x)\},$ where $k=(\tau, h).$
\end{lemma}
\begin{proof}
Let $(g, \phi) \in S^*\{(g,x)\}.$
From  $\mathrm{p} \circ H =  \eta \circ \mathrm{p}$, we have
\[
	H(g, \phi) = (t', \eta(\phi))
\]
for some $t' \in G'.$ Because $\mathrm{ev} \circ H =  k \circ \mathrm{ev},$ we obtain
\[
	(t', \eta(\phi)(t')) = k(g, \phi(g)) = (\tau(g, \phi(g)), h(g, \phi(g))).
\]
Therefore, $\eta(\phi)(\tau(g, x)) = \hat h(x).$
\end{proof}

\begin{theorem}\label{hom_thm}
Let the star-constructions $S^*W$ and $(S')^* W'$ be isomorphic via phase space-preserving isomorphisms, where $W$ and $W'$ are open. Then, $\hat W$ and $\hat W'$ are homeomorphic. Further, $W_x$ is homeomorphic to $W'_{\hat h(x)}$ for each $x \in \hat W.$ 
\end{theorem}
\begin{proof}
Let the canonical projection $p_X: W \to \hat W$ be defined by $p_X(g,x) : = x.$ Similarly, we define the canonical projection $p_X': W' \to \hat W'.$ If we set $k^{-1} = (\tilde \tau, \tilde h)$, we have the following commutative diagram, and it is immediately clear that $\hat h: \hat W \to \hat W'$ is a homeomorphism:
\[
\begin{diagram}
\node{W} 
\arrow{e,t}{k}
\arrow{s,r}{p_X}
\arrow{se,t}{h} 
\node{W'}
\arrow{s,r}{p_X'}
\arrow{e,t}{k^{-1}}
\arrow{se,t}{\tilde h}
\node {W}
\arrow{s,r}{p_X} 
\\
\node {\hat W} 
\arrow{e,b}{\hat{h}} 
\node{\hat W'} 
\arrow{e,b}{\hat{\tilde{h}}} 
\node{\hat W}
\end{diagram}
\]
Now we show that $W_x$ is homeomorphic to $W'_{\hat h(x)}$ for each $x \in \hat W.$ We define $\tau_x : W_x \to W'_{\hat h(x)}$ by $\tau_x(g) := \tau(g,x).$ This is well-defined and continuous because
\[
	k(g,x) = (\tau(g,x), \hat h(x)) \in W'
\]
for each $g \in W_x.$ To show that $\tau_x$ is a homeomorphism, it suffices to confirm that $k(W_x \times \{x\}) = W'_{\hat h(x)} \times \{\hat h(x)\}$ because it implies $\tau_x \circ \tilde\tau_{\hat h(x)} = \mathrm{id}_{W'_{\hat h(x)} }$ and $\tilde\tau_{\hat h(x)} \circ \tau_x = \mathrm{id}_{W_x}.$ By definition, we have
\[
	\begin{aligned}
		&k(W_x \times \{x\}) \subset W'_{\hat h(x)} \times \{\hat h(x)\}\\
		&k^{-1}(W'_{\hat h(x)} \times \{\hat h(x)\}) \subset W_x \times \{x\}.
	\end{aligned}
\] 
Therefore, $k(W_x \times \{x\}) = W'_{\hat h(x)} \times \{\hat h(x)\}.$
\end{proof}

The notion of isomorphism via phase space-preserving isomorphisms respects basic dynamical properties.
\begin{theorem}\label{inv_orb}
Let the star-constructions $S^*(G\times X)$ and $(S')^*(G'\times X')$ be isomorphic via phase space-preserving isomorphisms. Then we have
\[
	\hat h\left(\mathcal{O}(\phi) \right) = \mathcal{O}\left( \eta(\phi)\right)
\]
for all $\phi \in S.$
\end{theorem}
\begin{proof}

Let $\phi \in S.$ If $y \in \hat h\left(\mathcal{O}(\phi) \right),$ we may find $g \in \mathrm{dom}(\phi)$ such that $y = \hat h(\phi(g))$ by definition. From Lemma \ref{iv_lem}, we have
\[
	\eta(\phi)(\tau(g, \phi(g))) = \hat h(\phi(g)).
\]
Therefore, $y \in \mathcal{O}\left( \eta(\phi)\right).$ 

Conversely, let $y \in \mathcal{O}\left( \eta(\phi)\right).$ By definition, there exists $g' \in \mathrm{dom}\left( \eta(\phi)\right)$ such that  $y = \eta(\phi)(g').$ Therefore, $(g', \eta(\phi)) \in (S')^*\{(g',y)\}.$ Let $(g_0, \phi) := H^{-1}(g', \eta(\phi)).$ Then we have $\tau(g_0, \phi(g_0)) = g',$ and 
\[
	y = \eta(\phi)(g')= \eta(\phi)(\tau(g_0, \phi(g_0))) = \hat h(\phi(g_0)).
\]
Therefore, $y \in \hat h\left(\mathcal{O}(\phi) \right).$ 
\end{proof}
\begin{corollary}
Let the star-constructions $S^*(G\times X)$ and $(S')^*(G'\times X')$ be isomorphic via phase space-preserving isomorphisms, where $S$ and $S'$ are $\sigma$-invariant. Then $A \subset X$ is weakly invariant if and only if $\hat h (A)$ is weakly invariant. In particular, $x_0 \in X$ is an equilibrium if and only if $\hat h(x_0)$ is an equilibrium.
\end{corollary}
\begin{proof}
Because we may consider the inverse of the morphism, it suffices to show that $A$ is weakly invariant if $\hat h(A)$ is so.

Let $x \in A.$ Because $\hat h(x) \in \hat h(A)$ and $\hat h(A)$ is weakly invariant, we may find $\psi \in S'$ such that $\psi(e) = \hat h(x)$ and $\mathcal{O}(\psi) \subset \hat h(A).$ Therefore, we have
\[
	\hat h\left(\mathcal{O}(\eta^{-1}(\psi)) \right) = \mathcal{O}\left( \psi \right) \subset \hat h(A)
\]
by Theorem \ref{inv_orb}. Because $\hat h$ is a homeomorphism, $\mathcal{O}(\eta^{-1}(\psi)) \subset A.$ 

By Lemma \ref{lem_iv_ch}, the proof is complete if we show $x \in \mathcal{O}(\eta^{-1}(\psi)).$ This follows from
\[
	\hat h(x) = \psi(e) \in \mathcal{O}(\psi) = \hat h\left(\mathcal{O}(\eta^{-1}(\psi)) \right).
\]
Therefore, $A$ is weakly invariant.

From Theorem \ref{inv_orb}, it is clear that $x_0 \in X$ is an equilibrium if and only if $\hat h(x_0)$ is an equilibrium.
\end{proof}

Not only are the orbits preserved, but their parametrization is also in good correspondence.
\begin{theorem}\label{isotopy_thm}
Let the star-constructions $S^*(G\times X)$ and $(S')^*(G'\times X')$ be isomorphic via phase space-preserving isomorphisms. For each $\phi \in S$, we define a map $\mathcal{D}_\phi : \mathrm{dom}(\phi) \to \mathrm{dom}(\eta(\phi))$ by $\mathcal{D}_\phi(g) := \tau(g, \phi(g)).$ Then, $\mathcal{D}_\phi $ is a homeomorphism.

Further, if $S$ and $S'$ have domains $G$ and $G',$ the map $\mathcal{D}: S \times G \to G'$ defined by $\mathcal{D}(\phi, g) := \mathcal{D}_\phi(g)$ is continuous. In particular, if $S$ is path connected, then all maps in the family $\{ \mathcal{D}_\phi : G \to G' \mid \phi \in S \}$ are isotopic.

\end{theorem}
To prove this theorem, we observe the following property of the map $\mathcal{D}_\phi.$
\begin{lemma}\label{comp_dom_lem}
Let the star-constructions $S^*(G\times X),$ $(S')^*(G'\times X')$ and $(S'')^*(G''\times X'')$ be isomorphic via phase space-preserving isomorphisms $<H, k, \eta>: S^*(G\times X)\to (S')^*(G'\times X')$ and $<H', k', \eta'>:(S')^*(G'\times X') \to (S'')^*(G''\times X'')$, where $k=(\tau, h)$ and $k' =(\tau', h')$.  

If $k'' := k' \circ k = (\tau'', h''),$ we have
\[
	\mathcal{D}_{\phi}'' = \mathcal{D}_{\eta(\phi)}' \circ \mathcal{D}_\phi,
\]
where $\mathcal{D}_{\eta(\phi)}' (g') := \tau'(g', \eta(\phi)(g'))$ and $\mathcal{D}_{\phi}'' (g) := \tau''(g, \phi(g)).$ 
\end{lemma}
\begin{proof}
Using Lemma \ref{iv_lem}, we have
\[
\begin{aligned}
	\mathcal{D}_{\phi}''(g) &=  \tau'\left(\tau(g, \phi(g)), \hat h(\phi(g))\right)\\
					&= \tau'\left(\tau(g, \phi(g)), \eta(\phi)(\tau(g, \phi(g)))\right) \\
					&= \mathcal{D}_{\eta(\phi)}' \circ \mathcal{D}_\phi(g),
	\end{aligned}
\]
for each $g \in \mathrm{dom}(\phi).$
\end{proof}

\begin{proof}[Proof of Theorem \ref{isotopy_thm}]
From Lemma \ref{iv_lem}, $\mathcal{D}_\phi$ is well-defined and its continuity is obvious. To show that it is a homeomorphism, it suffices to show that $(\mathcal{D}_\phi)^{-1} = \mathcal{D}_{\eta(\phi)}',$ but this follows immediately from Lemma \ref{comp_dom_lem}.

If $S$ and $S'$ have domains $G$ and $G',$ the continuity of $\mathcal{D}$ follows from that of the evaluation map because we have $\mathcal{D}(\phi, g) = \tau(g, \phi(g)) = \tau(\mathrm{ev}(g, \phi)).$
\end{proof}
An example of a path connected set $S \subset C_s(G,X)$ is given by the following lemma.
\begin{lemma}
Let $\Phi : G \times X \to X$ be a $G$-action and $X$ be path connected and locally compact. Then, $S = \{\Phi(\cdot,x) \mid x \in X\}$ is path connected.
\end{lemma}
\begin{proof}
Let $\phi_0, \phi_1 \in S,$ where $\phi_i = \Phi(\cdot, x_i)$ for $i=0,1.$ Because $X$ is path connected, there exists a continuous map $\gamma: [0,1] \to X.$ Then, $\Gamma(s) := \Phi(\cdot, \gamma(s))$ is the desired path between $\phi_0$ and $\phi_1.$
\end{proof}

A change of variables yields a phase space-preserving isomorphism.
\begin{theorem}
Let $S \subset C_s(G,X)$ and $h: X \to X'$ and $\tau: G \to G'$ be homeomorphisms. If we set
\[
	S' := \{ h \circ \phi \circ \tau^{-1} \mid \phi \in S\} \subset C_s(G,X),
\]
then $S^*(G\times X)$ and $(S')^*(G' \times X')$ are isomorphic via phase space-preserving isomorphisms.
\end{theorem}

\begin{proof}
We define a morphism $<H,k, \eta>$ by setting
\[
	\begin{aligned}
		k(g,x) &:= (\tau(g), h(x))\\
		\eta(\phi) &:= h \circ \phi \circ \tau^{-1}\\
		H(g, \phi) &:= (\tau(g), \eta(\phi))
	\end{aligned}
\]
for each $x\in X,$ $g \in G,$ and $(g, \phi) \in S^*(\{(g,x)\}).$ It is clear that the map $k$ is a homeomorphism. For $\eta,$ we have $\eta^{-1}(W(K,V)) = W(\tau^{-1}(K), h^{-1}(V))$ for all compact $K \subset G'$ and open $V \subset X'.$ Therefore, $\eta$ is also a homeomorphism.
\end{proof}


Considering Theorem \ref{isotopy_thm}, we may define a generalization of the concept of topological equivalence and topological conjugacy.

\begin{definition}[Topological equivalence]
Let $S \subset C_s(G,X)$, $S' \subset C_s(G,X'),$  where $S$ and $S'$ have the domain $G$.  
 $S$ and $S'$ are \emph{topologically equivalent} if $S^*(G\times X)$ and $(S')^*(G\times X')$ are isomorphic via phase space-preserving isomorphisms and these morphisms can be taken so that $\mathcal{D}_\phi$ is isotopic to the identity of $G$ and $\mathcal{D}_\phi(e) = e$ for all $\phi\in S.$

 $S$ and $S'$ are \emph{topologically conjugate} if $S^*(G\times X)$ and $(S')^*(G\times X')$ are isomorphic via phase space-preserving isomorphisms and these morphisms can be taken so that $\mathcal{D}_\phi$  is the identity on $G$ for all $\phi\in S.$

 \end{definition}
 \begin{lemma}
 Topological equivalence is an equivalence relation.
 \end{lemma}
 \begin{proof}
 This is a consequence of Remark \ref{rem_equiv}, Lemma \ref{comp_dom_lem}, and the property of isotopy class.
 \end{proof}
If $G=\mathbb{R}$ and $S$ is given by a flow,  these definitions relate to the usual one as follows.
\begin{lemma}\label{lem_monotone}
A continuous function $f : \mathbb{R} \to \mathbb{R}$ is isotopic to identity if and only if it is monotonically increasing and bijective.
\end{lemma}
\begin{proof}
If $f : \mathbb{R} \to \mathbb{R}$ is isotopic to identity, it is clear from the definition that $f$ is monotonically increasing and bijective.

Conversely, if  $f$ is monotonically increasing and bijective, the map $H:[0,1] \times \mathbb{R} \to \mathbb{R}$ defined by
\[
	H(s, x) := s f(x) + (1-s) x
\]
is an isotopy between $f$ and the identity.
\end{proof}
\begin{lemma}\label{lem_inv}
Let $f:X\times Y \to Z$ be a continuous map.
If $f(x, \cdot):Y \to Z$ is a homeomorphism for each $x \in X,$ there exists a continuous map $g: X \times Z \to Y$ satisfying
\[
	f(x, g(x, z)) = z
\]
and
\[
	g(x,f(x,y)) = y
\]
for all $x \in X$, $y\in Y$ and $z \in Z.$
\end{lemma}
\begin{proof}
First we show that $f$ is an open map. Let $W \subset X\times Y$ be open and take $z \in f(W).$ By definition, we may find $(x,y) \in X\times Y$ with $f(x,y) =z.$ Let $U\subset X$ and $V \subset Y$ be open neighborhoods of $x$ and $y$ such that $U \times V \subset W.$ Then we have
\[
	z \in f(U \times V) = \bigcup_{u \in U} f(u, V) \subset f(W).
\]
Because $f(u,\cdot)$ is a homeomorphism for each $u \in U,$ $f(U \times V) $ is open. Therefore, $f$ is an open map.

Now we define $\bar f(x,y) := (x, f(x,y)).$ Then it is open, continuous and bijective. Therefore, $\bar f$ is a homeomorphism. Therefore, we may find a continuous map $g: X \times Z \to Y$ satisfying
\[
	f(x, g(x, z)) = z
\]
and
\[
	g(x,f(x,y)) = y
\]
for all $x \in X$, $y\in Y$ and $z \in Z.$
\end{proof}
 \begin{theorem}
Let the star-constructions $S^*(\mathbb{R}\times X)$ and $(S')^*(\mathbb{R}\times X')$ be described by flows, that is, there exist flows $\Phi : \mathbb{R} \times X \to X$ and $\Phi : \mathbb{R}\times X' \to X'$ such that
\[
	\begin{aligned}
		S &= \{\Phi(\cdot,x) \mid x \in X\}\\
		S' &=\{\Psi(\cdot,x') \mid x' \in X'\}.
	\end{aligned}
\]
Assume that $X$ and $X'$ are locally compact. Then $S$ and $S'$ are topologically equivalent if and only if there exist a homeomorphism $h: X \to X'$ and a continuous map $\tau: \mathbb{R}\times X \to \mathbb{R}$ such that
\[
	\Psi(\tau(t,x), h(x)) = h(\Phi(t,x))
\] 
for all $(t,x) \in \mathbb{R}\times X,$ where $\tau(\cdot, x)$ is monotonically increasing and bijective, and $\tau(0,x) = 0$ for all $x \in X$.

Also, $S$ and $S'$ are topologically conjugate if and only if there exists a homeomorphism $h: X \to X'$ such that
\[
	\Psi(t, h(x)) = h(\Phi(t,x))
\] 
for all $(t,x) \in \mathbb{R}\times X.$ 
 \end{theorem}
 \begin{proof}
Let $S$ and $S'$ be topologically equivalent. Let a map $\mathcal{D}: S \times \mathbb{R} \to \mathbb{R}$ be that in Theorem \ref{isotopy_thm}. We define
\[
	\bar \tau(t, x) : = \mathcal{D}(\Phi(\cdot, x), t).
\]
From Lemma \ref{iv_lem}, we have
\[
	\eta(\Phi(\cdot, x))(\bar \tau(t, x)) = \hat h(\Phi(t, x)).
\]
By considering the case $t=0,$ we deduce
\[
	\eta(\Phi(\cdot, x)) = \Psi(\cdot, \hat h(x)),
\]
which implies
\[
	\Psi(\bar \tau(t, x), \hat h(x)) = \hat h(\Phi(t, x)).
\]
Because $\mathcal{D}(\Phi(\cdot, x), \cdot)$ is isotopic to the identity,  $\bar\tau(t, x)$ is monotonically increasing and bijective for all $x \in X$ by Lemma \ref{lem_monotone}. Also, we have $\bar\tau(0, x) = x$ for all $x \in X$ by definition. The map $\hat h: X \to X'$ is a homeomorphism by Theorem \ref{hom_thm}.

Conversely, let there exist a homeomorphism $h: X \to X'$ and a continuous map $\tau: \mathbb{R}\times X \to \mathbb{R}$ such that
\[
	\Psi(\tau(t,x), h(x)) = h(\Phi(t,x))
\] 
for all $(t,x) \in \mathbb{R}\times X,$ where $\tau(\cdot, x)$ is monotonically increasing and bijective, and $\tau(0,x) = 0$ for all $x \in X$.

We define a morphism $<H,k,\eta>: S^*(\mathbb{R}\times X) \to (S')^*(\mathbb{R}\times X')$ by
\[
	\begin{aligned}
		\eta(\Phi(\cdot, x)) &:= \Psi(\cdot, h(x))\\
		k(t,x) &= (\tau'(t,x), h(x))\\
		H(t, \Phi(\cdot, x)) &= (\tau(t,x), \eta(\Phi(\cdot, x))),
	\end{aligned}
\]
where $\tau'(t,x): = -\tau(-t, x).$ Because $\eta^{-1}(\Psi(\cdot, y)) = \Phi(\cdot, h^{-1}(x)),$ $\eta$ is a homeomorphism. By applying Lemma \ref{lem_inv} to $f(y, t):= \tau'(t, h^{-1}(y))$, we see that $k$ is a homeomorphism. Therefore, $S^*(\mathbb{R}\times X)$ and $(S')^*(\mathbb{R}\times X')$ are isomorphic via a phase space-preserving isomorphism. For the isotopy property, let us show that $\mathcal{D}_{\Phi(\cdot,x)}(t) =-\tau(-t, \Phi(t,x))$ is monotonically increasing and bijective for all $t$. If $x$ is an equilibrium, this is obvious. If $x$ is not an an equilibrium, there are two possibilities. If $\mathcal{O}(\Phi(\cdot, x))$ has no self-intersection, $\Psi(\cdot, h(x))$ is injective. In this case we have
\[
	 \tau(-t, \Phi(t,x)) + \tau(t,x) = 0,
\]
which follows from
\[
	\begin{aligned}
		\Psi(\tau(-t, \Phi(t,x)) + \tau(t,x), h(x)) &= \Psi(\tau(-t, \Phi(t,x)), h(\Phi(t,x)))\\
									&= h(\Phi(-t,\Phi(t,x))) \\
									&= h(x). 
	\end{aligned}
\]
Therefore, $\mathcal{D}_{\Phi(\cdot,x)}(t) = \tau(t,x).$ If $\mathcal{O}(\Phi(\cdot, x))$ has a self-intersection, $\Psi(\cdot, h(x))$ is periodic. Let $T$ be the period. Then we have
\[
	\tau(-t, \Phi(t,x)) + \tau(t,x) = n(t) T
\]
for some $n(t) \in \mathbb{Z}.$ Because $n(t) = (\tau(-t, \Phi(t,x)) + \tau(t,x))/T$ is a continuous function that takes integer values, it is a constant function. By considering $t =0,$ we have $n(t) =0.$ Therefore, we have $\mathcal{D}_{\Phi(\cdot,x)}(t) = \tau(t,x)$ also in this case.

The proof for the case of topological conjugacy is similar.
 \end{proof}
 \begin{remark}
 Thus, our definition of topological equivalence is stronger than the usual one. Indeed, in the usual definition, the continuity of $\tau$ is not guaranteed. However, there are examples where $\tau$ can be retaken so that it becomes continuous. It is not clear whether such regularization is always possible.
 \end{remark}

\section{Applications}
In this section, we apply the formalism given in the preceding sections to some examples to illustrate its use and properties.

\begin{example}
The solution sets of the following systems are not isomorphic.
\begin{enumerate}
	\item $\dot x \in [1/2,1]$.
	\item $\dot x \in \{1/2,1\}$.
	\item $\dot x = 1.$
\end{enumerate}
Indeed,  (1) and (3) satisfy the compactness axiom, while inclusion (2) does not. The uniqueness axiom is satisfied only in equation (3). From Theorem \ref{iso_inv}, these axioms are invariant under isomorphisms. Therefore, these three examples are not isomorphic.
\end{example}
\begin{proof}
It is an immediate consequence of the Arzel\`a-Ascoli theorem that the solution set of inclusion (1) satisfies the compactness axiom. For inclusion (2), consider a sequence $\{\phi_n\}$ of solutions defined inductively by
\[
	\phi_0(t) = \begin{cases}
				t + \frac{1}{4} &(t< -\frac{1}{2})\\
				\frac{1}{2} t &(-\frac{1}{2} \leq t < \frac{1}{2})\\
				t - \frac{1}{4} &(\frac{1}{2} \leq t)\\
			\end{cases}
\]
and
\[
	\phi_{n+1}(t) = \begin{cases}
				\frac{1}{2} \left( \phi_n(2 t + 1) - \frac{3}{4} \right) &(t< 0)\\
				\frac{1}{2} \left( \phi_n(2 t - 1) + \frac{3}{4} \right) &(0 \leq t).\\
			\end{cases}
\]
In $C(\mathbb{R}, \mathbb{R}),$ $\{\phi_n\}$  converges to $\psi$ given by
\[
	\psi(t) = \begin{cases}
				t + \frac{1}{4} &(t< -1)\\
				\frac{3}{4} t &(-1 \leq t <1)\\
				t - \frac{1}{4} &(1 \leq t),\\
			\end{cases}
\]
which is not a solution of inclusion (2). By Lemma \ref{lem_conv_unique}, it follows that no subsequence of $\{\phi_n\}$ converges to a solution of inclusion (2) in $C_s(\mathbb{R}, \mathbb{R})$. From Theorem \ref{kamke}, this implies that the solution set of inclusion (2) does not satisfy the compactness axiom.

It is clear that the uniqueness axiom is satisfied only in equation (3).
\end{proof}
\begin{example}
Let us consider the following family of equations with a parameter $a \in \mathbb{R}:$
\begin{equation}\label{eq_sn}
	x' = x^2 +a.
\end{equation}
Let the corresponding solution set be $S_a.$ Then $S_a^* \mathbb{R}^2$ and $S_b^* \mathbb{R}^2$ are isomorphic via a phase space-preserving isomorphism if $a$ and $b$ have the same sign. 
\end{example}
\begin{proof}
The solution $\phi^a_{(t_0,x_0)}$ of (\ref{eq_sn}) satisfying the initial condition $\phi^a_{(t_0,x_0)}(t_0) = x_0$ is given by 
\[
	\phi^a_{(t_0,x_0)}(t) = \begin{cases}
				\sqrt{a} \tan\left( \sqrt{a}(t-t_0) + \tan^{-1}(x_0/\sqrt{a}) \right) & (a >0)\\
				-\sqrt{-a} \tanh\left( \sqrt{-a}(t-t_0) + \tanh^{-1}(-x_0/\sqrt{-a}) \right) &(a <0, |x_0| < \sqrt{-a})\\
				\pm \sqrt{-a} & (a <0, x_0 =\pm\sqrt{-a})\\
				\frac{-\sqrt{-a}}{ \tanh\left( \sqrt{-a}(t-t_0) + \tanh^{-1}(-\sqrt{-a}/x_0) \right) }&(a <0, |x_0| > \sqrt{-a})\\
			\end{cases}
\]
From these expressions, we have
\[
	\mathrm{dom}(\phi^a_{(t_0,x_0)}) = \left(t_0-\frac{1}{\sqrt{a}}\tan^{-1}\frac{x_0}{\sqrt{a}}-\frac{\pi}{2\sqrt{a}},t_0-\frac{1}{\sqrt{a}}\tan^{-1}\frac{x_0}{\sqrt{a}}+ \frac{\pi}{2\sqrt{a}}\right)
\]
for $a> 0$ and 
\[
\mathrm{dom}(\phi^a_{(t_0,x_0)}) = \begin{cases}
		\mathbb{R} & (|x_0| \leq \sqrt{-a}) \\
		(-\infty,  t_0 - \frac{1}{\sqrt{-a}}\tanh^{-1}(-\sqrt{-a}/x_0)) & (x_0 > \sqrt{-a}) \\
		(t_0 - \frac{1}{\sqrt{-a}}\tanh^{-1}(-\sqrt{-a}/x_0), \infty) & (x_0 < \sqrt{-a})
	\end{cases}
\]
for $a<0.$ In either case, the maps
\[
\begin{aligned}
	\eta(\phi^a_{(t_0,x_0)}) &:= \phi^b_{(\sqrt{a/b}t_0,\sqrt{b/a}\,x_0)}\\
	k(t,x) &:= \left( \sqrt{\frac{a}{b}}t  , \sqrt{\frac{b}{a}}x\right)\\
	H(t, \phi)& = \left( \sqrt{\frac{a}{b}}t , \eta(\phi)\right)
\end{aligned}
\]
give a phase space-preserving isomorphism between $S_a^* \mathbb{R}^2$ and $S_b^* \mathbb{R}^2.$ It can be verified directly that the domains are also preserved, as in Theorem \ref{isotopy_thm}.
\end{proof}
%
\begin{example}
The solution sets of the following systems are topologically conjugate.
\begin{enumerate}
	\item $\dot x \in [1/2,1]$.
	\item $\dot x \in [1, 2].$
	\item $\dot x \in [-1,-1/2]$.
\end{enumerate}
\end{example}
\begin{proof}
Let the solution set of (1) be $S_1$, that of (2) be $S_2$, and that of (3) be $S_3$. We define an isomorphism between $S_1^*\mathbb{R}^2$ and $S_2^*\mathbb{R}^2$ by
\[
\begin{aligned}
\eta_1(\phi) &:= 2 \phi\\
k_1(t,x) &:= (t, 2 x)\\
H_1(t,\phi) &:= (t, \eta_1(\phi)).
\end{aligned}
\]
Because $\mathcal{D}_{\phi}(t) = t$ for all $\phi \in S_1,$ $S_1$ and $S_2$ are topologically conjugate. 

For $S_1^*\mathbb{R}^2$ and $S_3^*\mathbb{R}^2,$ we define
 \[
\begin{aligned}
\eta_2(\phi) &:= - \phi\\
k_2(t,x) &:= (t, - x)\\
H_2(t,\phi) &:= (t, \eta_2(\phi)).
\end{aligned}
\]
Then $S_1$ and $S_3$ are topologically conjugate. 
\end{proof}
\begin{example}
For a compact metric space $X$, let us consider the group of automorphisms $\mathrm{Aut}(X)$ and the space $C(X)$ of all continuous functions on $X,$ each topologized by the compact-open topology. We define $S_X := C(\mathrm{Aut}(X),C(X))$ and $W_X := \mathrm{Aut}(X) \times C(X).$

If $X$ and $Y$ are homeomorphic, $S_X^* W_X$ and $S_Y^* W_Y$ are isomorphic via a phase space-preserving isomorphism.
\end{example}
\begin{proof}
Let $h: X \to Y$ be a homeomorphism. Then we have a natural isomorphism $\tau : \mathrm{Aut}(X) \to \mathrm{Aut}(Y)$ defined by $\tau(g) := h\circ g \circ h^{-1}$ and the pullback $h^* : C(Y) \to C(X),$ which is a homeomorphism. 

We define the maps
\[
\begin{aligned}
	\eta(\phi) &:= (h^*)^{-1}\circ\phi \circ \tau^{-1}\\
	k(g,f) &:= ( \tau(g), (h^*)^{-1}(f) )\\
	H(g, \phi)& = \left( \tau(g), \eta(\phi)\right),
\end{aligned}
\]
 for each $g \in  \mathrm{Aut}(X)$, $f \in C(X)$ and $\phi \in S_X.$ Using the continuity of compositions, it can be shown that these three maps are homeomorphisms. Thus, $S_X^* W_X$ and $S_Y^* W_Y$ are isomorphic via a phase space-preserving isomorphism.
\end{proof}

\section{Concluding remarks}

As the purpose of this article is to present a basic framework, further study on its application to various problems is necessary. In particular, bifurcation or structural stability can be studied once we have introduced a proper notion of the ``distance" between systems.

Regarding the dynamical behavior, our discussion here has been limited to weakly invariant sets and equilibrium points.  However, it appears possible to generalize other notions too.

\section*{Acknowledgements}
This study was supported by a Grant-in-Aid for JSPS Fellows
(20J01101). 
\bibliographystyle{plain}
\bibliography{refs}
\end{document}